\documentclass[reqno]{amsart}
\usepackage{amsmath,amssymb,amsthm,bm}

\def\bbbq{\mathbb Q}
\def\bbbc{\mathbb C}
\def\bbbf{\mathbb F}
\def\bbbz{\mathbb Z}

\def\bbbg{\mathbb G}

\def\bolda{\bm{a}}
\def\boldb{\bm{b}}

\def\det{{\rm det}}

\def\rank{{\rm rank}}
\def\v#1{{\bf #1}}

\newtheorem{theorem}[subsection]{Theorem}
\newtheorem{proposition}[subsection]{Proposition}
\newtheorem{definition}[subsection]{Definition}
\newtheorem{lemma}[subsection]{Lemma}
\newtheorem{remark}[subsection]{Remark}

\newtheorem{corollary}[subsection]{Corollary}

\begin{document}
\title[Duality relations for hypergeometric series]{Duality relations for hypergeometric series}
%\author{Frits~Beukers and Fr\'ed\'eric~Jouhet}

\author[Frits Beukers]{Frits Beukers}
\address{Utrecht University, Department of Mathematics, P.O. Box 80.010, 3508 TA Utrecht, Netherlands}
\email{f.beukers@uu.nl}
\urladdr{http://www.staff.science.uu.nl/{\textasciitilde}beuke106/}

\author[Fr\'ed\'eric Jouhet]{Fr\'ed\'eric Jouhet}
\address{Institut Camille Jordan, Universit\'e Claude Bernard Lyon 1,
69622 Villeurbanne Cedex, France}
\email{jouhet@math.univ-lyon1.fr}
\urladdr{http://math.univ-lyon1.fr/{\textasciitilde}jouhet/}

\thanks{}

\date{}

\subjclass[2010]{}


%\begin{document}

\begin{abstract}
We  explicitly give the relations between the hypergeometric solutions of the general
hypergeometric equation and their duals, as well as similar relations for q-hypergeometric
equations. They form a family of very general identities for hypergeometric series. Although
they were foreseen already by N. M. Bailey in the 1930's on analytic grounds, we give a purely
algebraic treatment based on general principles in general differential and difference modules.
\end{abstract}

\maketitle

%%%%%%%%%%%%%%%%%%%%%%%%%%%%%%%%%%%%%%%%%%%%%%%%%
\section{Introduction and notations}
%%%%%%%%%%%%%%%%%%%%%%%%%%%%%%%%%%%%%%%%%%%%%%%%%

Hypergeometric series in one variable were introduced by Euler and extensively studied
by Gauss~\cite[pp. 123, 207]{Ga1866} as solutions of the so-called hypergeometric differential
equation.
Later Thomae~\cite{Th1870} generalized the Gauss hypergeometric functions to higher order versions.
In this paper a one variable hypergeometric function of order $r$ depends on $2r$ complex parameters
$a_1,\ldots,a_r,b_1,\ldots,b_r$ where we take $b_r=1$. The
hypergeometric function is given by a power series in $z$ of the form
$$\sum_{n=0}^{\infty}{(a_1)_n\cdots(a_r)_n\over (b_1)_n\cdots(b_{r-1})_nn!}\ z^n.$$
Here $(x)_n$ stands for the Pochhammer symbol (or rising factorial) defined by $(x)_0:=1$ and
$$(x)_n:=x(x+1)\cdots(x+n-1),\quad n\ge1.$$
In order for the coefficients to exist we must assume that $b_i\not\in \bbbz_{\le0}$ for
all $i$. In that case the radius of convergence of the series is $1$, unless $a_i\in\bbbz_{\le0}$
for some $i$, in which case we have a polynomial. The corresponding function is denoted by
$${}_rF_{r-1}\!\left(\begin{matrix}a_1,\ldots,a_r\\
b_1,\ldots,b_{r-1}\end{matrix};z\right).$$
It satisfies the $r$-th order linear differential equation
\begin{equation}\label{eq:hypergeq}
(\theta+b_1-1)\cdots(\theta+b_r-1) f=z(\theta+a_1)\cdots(\theta+a_r)f,
\end{equation}
where $\theta=z{d\over dz}$. The singularities of this equation are given by $z=0,1,\infty$.
Around $z=0$ one can easily give a basis of solutions if one assumes that the $b_i$ (including $b_r=1$) are
distinct modulo $\bbbz$. They read
$$f_i(z):=z^{1-b_i}\,{}_rF_{r-1}\!\left(\begin{matrix}a_1+1-b_i,\ldots,a_r+1-b_i\\
b_1+1-b_i,\ldots,\vee,
\ldots,b_r+1-b_i\end{matrix};z\right),\qquad 1\leq i\leq r,$$
where $\vee$ denotes deletion of the term with index $i$.
Note that in this notation $f_r(z)$ is the hypergeometric function we started with.
For details concerning irreducibility, rigidity and monodromy of the equation we refer to the paper
\cite{BH89}.

Basic hypergeometric series, or q-hypergeometric series, have appeared in the eighteenth
century when Euler wrote the generating function for integer partitions as an infinite product.
However, the  systematic study of basic hypergeometric series and equations appeared quite later in Heine's paper~\cite{He}. We refer the reader to the classical books of  Gasper and Rahman~\cite{GR} and
Slater~\cite{Sl} for additional historical information.  The generalized basic hypergeometric equations
are quantizations (or $q$-analogues) of the previous generalized hypergeometric equations.

We recall some standard notation for basic (or q-) hypergeometric series
(see for instance~\cite{GR} for a
comprehensive study of their theory).
Let $q$ be a fixed complex parameter (the ``base'') with $0<|q|<1$.
The $q$-shifted factorial is defined for any complex
parameter $a$ and any non-negative integer $n$ by $(a;q)_0:=1$ and
$$
(a;q)_n:=(1-a)(1-aq)\cdots(1-aq^{n-1}),\qquad n\geq1.
$$

Further,  for a positive integer $r$ and complex  parameters $a_1,\dots,a_r$,  $b_1,\dots,b_{r}$ with $b_r=q$, recall the definition of the corresponding basic hypergeometric series,
$$
{}_r\phi_{r-1}\!\left[\begin{matrix}a_1,\dots,a_r\\
b_1,\dots,b_{r}\end{matrix};q,z\right]:=
\sum_{n=0}^\infty\frac{(a_1;q)_n\cdots(a_r;q)_n}{(b_1;q)_n\cdots(b_r;q)_n}z^n.
$$

Again, in order for the coefficients to exist we must assume that $b_i\not\in q^{\bbbz_{\le0}}$ for
all $i$. In that case the radius of convergence of the series is $1$, unless $a_i\in q^{\bbbz_{\le0}}$
for some $i$, in which case we have a polynomial. 
The $q$-hypergeometric series is a solution of the following $r$th order difference equation
\begin{equation}\label{eq:qhyperg}
(1-b_1\sigma_q/q)\cdots(1-b_{r}\sigma_q/q)f=z(1-a_1\sigma_q)\cdots(1-a_r\sigma_q)f,
\end{equation}
where $\sigma_q$ is the automorphism on the field of formal Laurent series $\bbbc((z))$ given by
$\sigma_q(f)(z):=f(qz)$. Note that replacing $a_i$ by $q^{a_i}$ and $b_i$ by $q^{b_i}$ for all $i$ and letting $q$ tends to $1$, the basic hypergeometric series tends to the previous hypergeometric series. Moreover, performing these replacements and dividing both sides of~\eqref{eq:qhyperg} by $(1-q^r)$ before letting $q$ tends to $1$, one gets back the hypergeometric differential equation~\eqref{eq:hypergeq}.

Suppose that no quotient of two different $b_i$'s (including $b_r=q$) is an integral power of $q$. Define for $i=1,\dots,r$ the
formal solution $z^{1-\beta_i}$ of $\sigma_q(f)=(q/b_i)f$. This means that $\beta_i$ can
be thought of as $\log(b_i)/\log(q)$.
It is then easy to show~\cite{Ja} that a basis of solutions of~\eqref{eq:qhyperg} reads
$$f_i(q;z):=z^{1-\beta_i}\,{}_r\phi_{r-1}\!\left[
\begin{matrix}qa_1/b_i,\ldots,qa_r/b_i\\qb_1/b_i,\ldots,\vee,
\ldots,qb_r/b_i\end{matrix};q,z\right],\qquad 1\leq i\leq r.$$

For the general theory of irreducibility and (strong) rigidity of the $q$-difference
equation~\eqref{eq:qhyperg}, we refer the reader to~\cite{Ro11, Ro14}.

In this paper we shall also consider the duals of the (basic) hypergeometric equations and
their solutions. For the differential case these solutions read
$$g_i(z):=z^{b_i-1}\,{}_rF_{r-1}\!\left(\begin{matrix}b_i-a_1,\ldots,b_i-a_r\\
b_i+1-b_1,\ldots,\vee,
\ldots,b_1+1-b_r\end{matrix};z\right),\qquad 1\leq i\leq r.$$

For the basic hypergeometric case, we introduce
$$g_i(q;z):=z^{\beta_i-1}\,{}_r\phi_{r-1}\!\left[
\begin{matrix}b_i/a_1,\dots,b_i/a_r\\qb_i/b_1,\dots,\lor,\dots,qb_i/b_r\end{matrix};q;\frac{a_1\dots a_rzq^{r-2}}{b_1\dots b_{r-1}}\right],\qquad 1\leq i\leq r.$$

The main goal of the present paper is to exhibit an explicit form of the duality relations
between the $f_i(z),g_i(z)$ and of the $q$-versions $f_i(q;z),g_i(q;z)$. Here is our first theorem.

\begin{theorem}\label{duality}
With the previous notations, set
$$
c_i:=\prod_{j=1\atop j\neq i}^r\frac{1}{b_j-b_i},
$$
for $i=1,\dots,r$, where $r\ge2$. Then we have for any $k,l=0,1,\ldots,r-1$,
$$\sum_{i=1}^rc_i\theta^k(f_i)(z)\,\theta^l(g_i)(z)=:M_{kl}\in H(z),$$
where $H$ is the field generated over $\bbbq$ by the $a_i,b_j$.

In particular we have
$M_{kl}=0$ if $k+l\le r-2$ and $M_{kl}={(-1)^{k}\over 1-z}$ if $k+l=r-1$.
\end{theorem}

As an example, the case $k=l=0,r=2$ yields
\begin{multline*}
{}_2F_1\left(\begin{matrix}a_1,a_2\\
b_1\end{matrix};z\right)
{}_2F_1\left(\begin{matrix}1-a_1,1-a_2\\
2-b_1\end{matrix};z\right)\\
=
{}_2F_1\left(\begin{matrix}a_1+1-b_1,a_2+1-b_1\\
2-b_1\end{matrix};z\right)
{}_2F_1\left(\begin{matrix}b_1-a_1,b_1-a_2\\
b_1\end{matrix};z\right).
\end{multline*}
This follows directly from Euler's standard identity
$${}_2F_1\left(\begin{matrix}a,b\\
c\end{matrix};z\right)=
(1-z)^{c-a-b}
{}_2F_1\left(\begin{matrix}c-a,c-b\\
c\end{matrix};z\right).$$
The case $k=l=0,r=3$ can be found in the paper \cite{Da32} by Darling,
published in 1932.
Very soon after, Bailey~\cite{Ba33} in 1933 found another method to prove this identity
and its generalizations. It is based on explicit calculation of the coefficients
of the sum $\sum_ic_if_ig_i$, which are then shown to vanish by a contour integration argument.
The approach we take in this paper is entirely different.
It is purely algebraic and based on the general relations that
arise from differential hypergeometric modules and their dual modules.

To give an idea of a complete set of relations we reproduce the matrix $M:=(M_{kl})_{0\leq k,l\leq r-1}$
for the case $r=2$ here,
$$M=\begin{pmatrix} 0 & {1\over1-z}\\
-{1\over 1-z} & {-2+b_1+b_2+(1-a_1-a_2)z\over (1-z)^2}
\end{pmatrix}
$$
and for the case $r=3$,
$$M=\begin{pmatrix} 0 & 0 & {1\over1-z}\\ 0 & -{1\over1-z} &
{3-e_1(\boldb)+(-2+e_1(\bolda))z\over (1-z)^2}\\
{1\over1-z} &  {-3+e_1(\boldb)+(1-e_1(\bolda))z\over (1-z)^2}& {A+Bz+Cz^2\over(1-z)^3}\\
\end{pmatrix},$$
where
\begin{eqnarray*}
A&=&(e_1(\boldb)-1)^2-e_2(\boldb)+2,\\
B&=&e_2(\boldb)-2(e_1(\bolda)-1/2)(e_1(\boldb)-5/2)+e_2(\bolda)-5/2,\\
C&=&(e_1(\bolda)-1)^2-e_2(\bolda),
\end{eqnarray*}
and
$$e_1(\v x)=x_1+x_2+x_3,\quad e_2(\v x)=x_1x_2+x_1x_3+x_2x_3$$
are elementary symmetric functions.

Such matrices can be computed with the formula $M=\Psi^{-1}$
coming from equation~\eqref{matrixduality}. 
There is an analogous result in the q-hypergeometric case.

\begin{theorem}\label{qduality}
With the previous notations, set
$$
c_i(q):=qb_i^{r-2}\prod_{j=1\atop j\neq i}^r\frac{1}{b_i-b_j},
$$
for $i=1,\dots,r$ where $r\ge2$. Then we have for any $k,l=0,1,\ldots,r-1$,
$$\sum_{i=1}^rc_i(q)\,f_i(q;zq^k)\,g_i(q;zq^l)=:M_{kl}(q)\in H_q(z),$$
where $H_q$ is the field generated by $q,a_i,b_j$ over $\bbbq$.

In particular we have $M_{kl}(q)=0$ if $l\le k\le r-1$, except when $(k,l)=(r-1,0)$. In the latter
case,
$$M_{r-1,0}(q)=\frac{(-1)^{r+1}q^r}{b_1\dots b_{r}-a_1\dots a_rq^{r-1}z}.$$
In addition we have
$$M_{k,k+1}(q)={1\over 1-q^kz},\quad k=0,1,\ldots,r-2.$$
\end{theorem}

For $k=l=0,r=2$ we obtain the relation
\begin{multline*}
{}_2\phi_1\!\left[
\begin{matrix}a_1,a_2\\b_1\end{matrix};q,z\right]
{}_2\phi_1\!\left[
\begin{matrix}q/a_1,q/a_2\\q^2/b_1\end{matrix};q;\frac{a_1a_2z}{b_1}\right]\\
={}_2\phi_1\!\left[
\begin{matrix}qa_1/b_1,qa_2/b_1\\q^2/b_1\end{matrix};q,z\right]
{}_2\phi_1\!\left[
\begin{matrix}b_1/a_1,b_1/a_2\\b_1\end{matrix};q;\frac{a_1a_2z}{b_1}\right],
\end{multline*}

where we explicitly set $b_2=q$.
Note that this identity can be proved by applying Heine's transformation formula for a ${}_2\phi_1$ series~\cite[Appendix III, (III.3)]{GR}. As communicated to us by S.~O.~Warnaar,
extracting coefficients of $z^n$ in this identity yields a special case of Sears' transformation
formula for a terminating balanced ${}_4\phi_3$ series~\cite[Appendix III, (III.16)]{GR}.
In~\cite{Ba33}, Bailey explained how to prove Theorem~\ref{qduality}
for $k=l=0$ and $r=3$ by using contour integration techniques. Moreover, the case $k=r-2, l=0$ and general $r$ is both proved in Sears' paper~\cite{Se51} by recursion and  classical techniques of $q$-series, and in Shukla's article~\cite{Sh57} by use of previous work due to Slater on bilateral basic hypergeometric series.

We did not find in the literature identities from Theorem~\ref{qduality} in full generality.
As in the differential case, our proof is purely algebraic and based on relations that
exist between q-hypergeometric difference modules and their duals. 

To give an idea of a complete set of relations we reproduce the matrix $M_q:=(M_{kl}(q))_{0\leq k,l\leq r-1}$
for the case $r=2$ here,
$$M_q=\begin{pmatrix}0 & {1\over 1-z}\\ {-q^2\over b_1b_2-a_1a_2qz} & 0
\end{pmatrix},$$
and the case $r=3$,
$$M_q=\begin{pmatrix}0 & {1\over 1-z} & {e_1(\v b)-e_1(\v a)qz\over (1-z)(1-qz)}\\
0 & 0 & {1\over 1-qz}\\
{q^3\over b_1b_2b_3-a_1a_2a_3q^2z} & 0 & 0
\end{pmatrix}
$$

Such matrices can be computed with the formula $M_q=\Psi_q^{-1}$
coming from equation~\eqref{matrixqduality}.

This paper is organized as follows. In Section~\ref{Dmodules},
we will recall the notion of D-modules associated to general linear differential equations and their corresponding duals.
The proof of Theorem \ref{duality} will then be given in Section~\ref{hypergeometric},
as a consequence of these general considerations. In a similar vein we recall the notion of difference modules associated to
general linear difference equations and their duals in Section~\ref{deltamodules}.
Theorem \ref{qduality} will then be proven
in Section~\ref{qhypergeometric}.

%%%%%%%%%%%%%%%%%%%%%%%%%%%%%%%%%%%%%%%%%%%%%%%%%
\section{D-modules and duality}\label{Dmodules}
%%%%%%%%%%%%%%%%%%%%%%%%%%%%%%%%%%%%%%%%%%%%%%%%%

Let $K$ be a differential field with a derivation
$D$. Let $K_0$ be the field of constants.

\begin{definition}
A D-module $M$ over $K$ is a $K$-vector space with a map
$\nabla:M\to M$ such that
\begin{enumerate}
\item $\nabla(m_1+m_2)=\nabla(m_1)+\nabla(m_2)$ for all $m_1,m_2\in M$,
\item $\nabla(fm)=D(f)m+f\nabla(m)$ for all $f\in K$ and $m\in M$.
\end{enumerate}
\end{definition}

In what follows we shall abbreviate $\nabla$ by $D$ again.
The major example is the D-module associated to a linear differential operator
$L\in K[D]$. Consider the left ideal $(L)$ generated by $L$.
The action of $D$ from the left on the quotient ring $K[D]/(L)$ turns it
into a D-module.

\begin{definition}\label{}
Let $M,M'$ be $D$-modules over $K$.

A $K$-linear map $\phi:M\to M'$ is called a $D$-homomorphism if
$D\circ \phi=\phi\circ D$. If $\phi$ is a $K$-vector space isomorphism
we call $\phi$ a $D$-isomorphism.
\smallskip

A $D$-module $M$ is called irreducible if $\{\v 0\}$ and $M$ are the
only D-submodules.
\smallskip

The tensor product $M\otimes M'$ can be given a $D$-module structure
via
$$D(m\otimes m')=D(m)\otimes m'+m\otimes D(m'),$$
for all $m\in M,m'\in M'$.

The dual $M^*$ of the vectorspace $M$ can be given a $D$-module structure
via
$$\langle D(m^*),m\rangle=D(\langle m^*,m\rangle)-\langle m^*,D(m)\rangle$$
for every $m^*\in M^*$ and $m\in M$. Here $\langle m^*,m\rangle$ denotes the evaluation
of $m^*$ at the point $m$.
\end{definition}

\begin{proposition}\label{explicitdual}
Let $M$ be a finite dimensional D-module over $K$. Let $m_1,\ldots,m_r$ be a
$K$-basis and let $A_{ij}\in K$ be such that
$$D(m_i)=\sum_{j=1}^rA_{ij}m_j$$
for all $i$. Define the dual basis $m^*_i$ by $\langle m^*_i,m_j\rangle=\delta_{ij}$ for
$i,j=1,\ldots,r$, where $\delta_{ij}$ is the Kronecker delta. Then,
\begin{enumerate}
\item $D(m^*_i)=-\sum_{j=1}^{r}A_{ji}m^*_j$ for $i=1,\ldots,r$,
\item $\omega=\sum_{i=1}^{r}m^*_i\otimes m_i\in M^*\otimes M$ does not depend
on the choice of basis $m_i$ and $D(\omega)=0$.
\end{enumerate}
\end{proposition}

\begin{proof}
To see the first assertion, evaluate the form $m_k^*$ on the equality
$D(m_i)=\sum_{j=1}^rA_{ij}m_j$. We get
$$\langle m_k^*,D(m_i)\rangle=\sum_{j=1}^r A_{ij}\langle m_k^*,m_j\rangle.$$
By definition of $D(m_k^*)$ and $\langle m_k^*,m_j\rangle=\delta_{jk}$ we get
$$-\langle D(m_k^*),m_i\rangle=A_{ik}.$$
From this it follows that $D(m_k^*)=-\sum_{i=1}^rA_{ik}m_i^*$.

The second assertion now follows by straightforward computation.
\end{proof}

In the following proposition we consider elements $\Omega\in N\otimes M$
where $N,M$ are D-modules of the same rank $r$. We say that $\Omega$ is
 non-degenerate if it cannot be written in the form $\sum_{i=1}^sn_i\otimes m_i$
with $m_i\in M,n_i\in N$ and $s<r$.

\begin{proposition}\label{correspondence}
Let $M,N$ be finite dimensional D-modules over $K$ of rank $r$.
Let $m_1,\ldots,m_r$ be a basis of $M$ and $m_1^*,\ldots,m_r^*$ the corresponding
dual basis of $M^*$.

Then the D-homomorphisms $\phi:M^*\to N$ are in one-to-one correspondence
with tensors $\Omega\in N\otimes M$ such that $D(\Omega)=0$.

The correspondence is given by $\phi$ goes to $\Omega=\sum_{i=1}^r
\phi(m_i^*)\otimes m_i$.

Conversely, the tensor $\Omega=\sum_{i=1}^rn_i\otimes m_i$ with
$D(\Omega)=0$ corresponds to the $K$-linear map generated by
$m_i^*\to n_i$ $(i=1,\ldots,r)$, which is a D-homomorphism.

Finally, $\phi$ is an isomorphism if and only if $\Omega$ is
non-degenerate.
\end{proposition}

\begin{proof}
Suppose we are given a $K$-linear map $\phi:M^*\to N$.
It is determined by its values $n_i:=\phi(m_i^*)$
for $i=1,\ldots,r$. We write the D-homomorphism condition in two ways.

First of all, by definition one should have $D(n_i)=D(\phi(m_i^*))=
\phi(D(m_i^*))$ for all $i$. Using Proposition \ref{explicitdual} this is the
same as
$$D(n_i)=\phi\left(-\sum_{j=1}^rA_{ji}m_j^*\right)=-\sum_{j=1}^rA_{ji}n_j, \quad i=1,\ldots,r.$$
Now define $\Omega=\sum_{i=1}^r n_i\otimes m_i$ and subsitute it into $D(\Omega)=0$.
We get
$$D(\Omega)=\sum_{i=1}^r D(n_i)\otimes m_i+n_i\otimes D(m_i)=0.$$
Using $D(m_i)=\sum_{j=1}^rA_{ij}m_j$ rewrite this as
$$\sum_{j=1}^r \left(D(n_j)+\sum_{i=1}^rA_{ij}n_i\right)\otimes m_j=0.$$
This is equivalent to $D(n_j)+\sum_{i=1}^rA_{ij}n_i=0$ for $j=1,\ldots,r$.
Here we recognize our formulation of the D-homomorphism condition.

Note that the tensor $\sum_{i=1}^rn_i\otimes m_i$ is non-degenerate if and
only if $n_1,\ldots,n_r$ are linearly independent. But this is equivalent to
$\phi$ being an isomorphism.
\end{proof}

The next result identifies (up to D-isomorphism) the dual differential operator
associated with any fixed differential operator $L\in K[D]$.

\begin{theorem}\label{theorem:dual}
Consider the $r$-th order differential operator
$$ L=A_rD^r+A_{r-1}D^{r-1}+\cdots+A_1D+A_0,$$
with $A_i\in K$ for all $i$. Then the dual of $K[D]/(L)$ is D-isomorphic
to $K[D]/(L^*)$ where
$$L^*=(-D)^r\circ A_r+(-D)^{r-1}\circ A_{r-1}+\cdots+(-D)\circ A_1+A_0.$$
\end{theorem}

\begin{proof}
According to Proposition \ref{correspondence} we need to find an
element $\Omega\in (K[D]/(L^*))\otimes
(K[D]/(L))$ such that $D(\Omega)=0$. We propose to take
$$\Omega=\sum_{k=1}^n[A_k,1]_{k-1}$$
where we define
$$[u,v]_m=u\otimes(D^m\circ v)-(D\circ u)\otimes(D^{m-1}\circ v) +
(D^2\circ u)\otimes (D^{m-2}\circ v)+\cdots+(-1)^m (D^m\circ u)\otimes v.$$
Note that
$$D([u,v]_m)=u\otimes(D^{m+1}\circ v)+(-1)^m (D^{m+1}\circ u)\otimes v.$$
Using this property we can show that $D(\Omega)=0$. The fact that $\Omega$
is non-degenerate follows from the invertibility of the matrix $\Psi$ defined
below.
\end{proof}

We get the following immediate consequence.
\begin{corollary}\label{corollary:omegafixed}
Let ${\mathcal K}$ be a differential extension of $K$ and suppose that
$f,g\in {\mathcal K}$ satisfy the equations $L(f)=0$ and $L^*(g)=0$, where $L$ and $L^*$ are defined in Theorem~\ref{theorem:dual}. Then
the element
\begin{eqnarray*}
\Omega(g,f)&:=&A_0gf\\
&&+(A_1g)D(f)-D(A_1g)f\\
&&+(A_2g)D^2(f)-D(A_2g)D(f)+D^2(A_2g)f\\
&&\vdots\\
&&+(A_ng)D^{n-1}(f)-D(A_ng)D^{n-2}(f)+\cdots+(-1)^{n-1}D^{n-1}(A_ng)f\\
\end{eqnarray*}
belongs to $K_0$, the subfield of constant elements of ${\mathcal K}$ under $D$.
\end{corollary}

The following lemma can be proved by direct computation.

\begin{lemma}
Let notations be as above. Then $\Omega(g,f)=\sum_{i,j=0}^{r-1}\tilde{A}_{ij}D^i(g)D^j(f)$
where
$$\tilde{A}_{ij}=(-1)^i\sum_{l=0}^{r-1-i-j}(-1)^j{l+i\choose i}D^l(A_{i+j+l+1}).$$
In particular we have that $\tilde{A}_{ij}=0$ if $i+j\ge r$ and
$\tilde{A}_{ij}=(-1)^iA_r$ if $i+j=r-1$.
\end{lemma}

Let $\Psi$ be the $r\times r$-matrix with the entries $\tilde{A}_{ij}$. Then $\Psi$ is
invertible, since its determinant is $A_r^r$. As a consequence $\Omega$ is non-degenerate.

Suppose we have a basis of solutions $f_1,\ldots,f_r$ of $L(f)=0$ and a basis
of solutions $g_1,\ldots,g_r$ of $L^*(g)=0$. Let us write $C_{ij}=\Omega(g_i,f_j)$
for all $i,j$ in $\{1,2,\ldots,r\}$ and denote the $r\times r$-matrix with these
entries by $C$. Let 
$$\bbbf:=W(f_1,\ldots,f_r)=
\begin{pmatrix}
f_1 & f_2 & \ldots & f_r\\
D(f_1) & D(f_2) & \ldots & D(f_r)\\
\vdots & \vdots & &\vdots\\
D^{r-1}(f_1) & D^{r-1}(f_2) & \ldots & D^{r-1}(f_r)
\end{pmatrix}$$ be the $r\times r$ Wronskian matrix and similarly $\bbbg:=W(g_1,\ldots,g_r)$. Then the previous matrix relation can be rewritten as
\begin{equation}\label{prematrixduality}
\bbbg^t\Psi\bbbf=C.
\end{equation}
Since the matrices $\bbbf$ and $\bbbg$ are Wronskian matrices of sets of independent functions,
they are invertible and we get
\begin{equation}\label{matrixduality}
\bbbf C^{-1}\bbbg^t=\Psi^{-1}.
\end{equation}
Notice that $\Psi^{-1}$ has zeros above the anti-diagonal which goes from the left lower corner
to the right upper corner. Moreover, by appropriately choosing the basis $g_1,\ldots,g_r$, we can assume that $C$ is the identity matrix, therefore yielding the following immediate consequence of~\eqref{matrixduality}.

\begin{corollary}\label{diagonalidentity}
Let $f_1,\ldots,f_r$ be a basis of solutions of $L(f)=0$, where $L$ is defined in Theorem~\ref{theorem:dual}. Then there
exists a solution basis $g_1,\ldots,g_r$ of $L^*(g)=0$ such that
$$\sum_{i=1}^rD^k(g_i) D^l(f_i) \in K$$
for all $k,l\ge0$. In particular we have $\sum_{i=1}^rD^k(g_i)D^l(f_i)=0$
for all $k,l$ with $k+l<r-1$.
\end{corollary}

%%%%%%%%%%%%%%%%%%%%%%%%%%%%%%%%%%%%%%%%%%%%%%%%%
\section{Hypergeometric equation of order $r$}\label{hypergeometric}
%%%%%%%%%%%%%%%%%%%%%%%%%%%%%%%%%%%%%%%%%%%%%%%%%

Recall the hypergeometric equation~\eqref{eq:hypergeq} of order $r$,
$$z(\theta+a_1)\cdots(\theta+a_r)f=(\theta+b_1-1)\cdots(\theta+b_r-1) f$$
where $\theta=z{d\over dz}$ and we have the default parameter $b_r=1$. By Theorem~\ref{theorem:dual}, the dual equation reads
$$(\theta-a_1)\cdots(\theta-a_r)(zg)=(\theta-b_1+1)\cdots(\theta-b_r+1) g.$$
Hence
$$z(\theta-a_1+1)\cdots(\theta-a_r+1)g=(\theta-b_1+1)\cdots(\theta-b_r+1) g.$$
So the dual equation is again hypergeometric with parameters $a'_i=1-a_i$ for
$i=1,\ldots,r$ and $b'_j=2-b_j$ for $j=1,\ldots,r$.

Suppose that the $b_j$ are all distinct modulo $\bbbz$.
Then recall from the introduction that the hypergeometric equation has a standard basis of solutions around $z=0$ of
the form
$$f_i(z)=z^{1-b_i}\,{}_rF_{r-1}\!\left(\begin{matrix}a_1+1-b_i,\ldots,a_r+1-b_i\\
b_1+1-b_i,\ldots,\vee,
\ldots,b_r+1-b_i\end{matrix};z\right),\qquad 1\leq i\leq r.$$
Similarly we have the following basis of solutions for the dual equation.
$$g_i(z)=z^{b_i-1}\,{}_rF_{r-1}\!\left(\begin{matrix}b_i-a_1,\ldots,b_i-a_r\\
b_i+1-b_1,\ldots,\vee,
\ldots,b_1+1-b_r\end{matrix};z\right),\qquad 1\leq i\leq r.$$

To use our D-module language, the ground field we employ here is the field of rational functions
$K=H(z)$ where $H$ is the field $\bbbq$ extended with the parameters $a_i,b_j$. The field
extension ${\mathcal K}$ in which our solutions lie is the field of Laurent series $H((z))$ extended
with the functions $z^{b_j}$. The latter are defined as non-trivial solution of $\theta(f)=b_j f$.

\begin{proposition}\label{dualpairinggauss}
Let $\Psi$ and $C$ be the $r\times r$-matrices defined in Section~\ref{Dmodules},
but for the case of the hypergeometric equation of order $r$ and $D=\theta$.
For the solution bases of the equation
and its dual we take the bases just defined. Then $C$ is a diagonal matrix. The $i$-th diagonal
entry reads
$$C_{ii}=\prod_{j=1\atop j\neq i}^r(b_j-b_i),$$
where $b_r=1$.
\end{proposition}

\begin{proof}
The elements of $C$ belong to $H$, the constant field. In particular, $C_{ij}\in H$ for all $i\ne j$.
However, $C_{ij}$ is also of the form $z^{-b_i+b_j}$ times a power series in $z$.
Since $b_i\ne b_j$ if $i\ne j$ we can only conclude that $C_{ij}=0$ for all
distinct $i,j$. For the computation of $C_{ii}$ it suffices to look only at
constant terms after cancellation of the powers $z^{1-b_i}$ and $z^{b_i-1}$
in formula~\eqref{prematrixduality}.
This means $C_{ii}$ equals
$$\left(1,b_i-1,\ldots,(b_i-1)^{r-1}\right)
\left(\begin{matrix}
B_1 & B_2 & \cdots     & B_r\\
-B_1 & -B_2 &  \cdots    & 0\\
B_2 & B_3 &    \cdots      & 0\\
\vdots &\vdots & &\vdots\\
(-1)^{r-1}B_{r-1} & (-1)^{r-1}B_r &     \cdots                  & 0\\
(-1)^rB_r & 0 & \cdots               & 0\end{matrix}\right)
\left(\begin{matrix}1\\ 1-b_i\\ \vdots\\ (1-b_i)^{r-1}\end{matrix}\right),
$$
where the $B_j$'s are the constants in the coefficients of the hypergeometric
operator. That is,
$$B_0+B_1\theta+\cdots+B_r\theta^r=(\theta-1+b_1)\cdots(\theta-1+b_r).$$
Direct calculation of $C_{ii}$ yields
$$C_{ii}=B_1+2B_2(1-b_i)+3B_3(1-b_i)^2+\cdots+rB_r(1-b_i)^{r-1}.$$
Note that $C_{ii}$ is simply the derivative of $(x-1+b_1)\cdots(x-1+b_r)$
evaluated at $x=1-b_i$. Hence $C_{ii}=\prod_{j\ne i}(b_j-b_i)$,
as asserted.
\end{proof}

As an immediate consequence of Corollary~\ref{diagonalidentity} and Proposition~\ref{dualpairinggauss}
we get the following result, which implies Theorem~\ref{duality} through explicitation of the desired terms in the matrix $\Psi^{-1}$ from equation~\eqref{matrixduality}.
\begin{corollary}\label{finalgauss}
Let $(f_i)_i$ and $(g_j)_j$ be the basis of solutions of the hypergeometric equation and the
dual equation, as given above. Let $H$ be the field generated over $\bbbq$ by  the $a_i,b_j$ and $q$. Then, with
$C_{ii}$ as in Proposition~\ref{dualpairinggauss}, we get
$$\sum_{i=1}^r{1\over C_{ii}}\theta^k(f_i)(z)\,\theta^l(g_i)(z)\in H(z),$$
for all integers $k,l\ge0$. In particular we have $\sum_{i=1}^r\theta^k(f_i)(z)\,\theta^l(g_i)(z)=0$
for all $k,l$ with $k+l<r-1$.
\end{corollary}

%%%%%%%%%%%%%%%%%%%%%%%%%%%%%%%%%%%%%%%%%%%%%%%%%
\section{$\Delta$-modules and duality}\label{deltamodules}
%%%%%%%%%%%%%%%%%%%%%%%%%%%%%%%%%%%%%%%%%%%%%%%%%

Let $K$ be a field of characteristic zero and $\Delta:K\to K$ a fixed isomorphism. We denote by $K_0:=\{a\in K|\Delta(a)=a\}$ the subfield of constants.
\begin{definition}\label{definition:deltamodules}
A $K$-vector space $M$ is called a $\Delta$-module (over $K$) if there
is a bijective map $\nabla:M\to M$
such that
\begin{enumerate}
\item $\nabla(m_1+m_2)=\nabla(m_1)+\nabla(m_2)$ for all $m_1,m_2\in M$,
\item $\nabla(fm)=\Delta(f)\nabla(m)$ for all $f\in K$ and $m\in M$.
\end{enumerate}
\end{definition}

We denote $\nabla$ by $\Delta$ again.
Consider the skew ring $K[\Delta,\Delta^{-1}]$ and an operator $L\in K[\Delta,\Delta^{-1}]$. Let
$(L)=\{\mu L|\mu\in K[\Delta,\Delta^{-1}]\}$ be the left ideal generated by $L$.
Then $K[\Delta,\Delta^{-1}]/(L)$
is again a $\Delta$-module, the module associated to the operator $L$. The action of $\Delta$
is given by left multiplication with $\Delta$. In order for
$\Delta$ to be bijective on this module we need that $L(0)\ne0$, which we will assume from now on.

\begin{definition}\label{definition:deltatensor}
Let $M,M'$ be $\Delta$-modules over $K$.

A $K$-linear map $\varphi:M\to M'$ is called a $\Delta$-homomorphism if
$\Delta\circ \varphi=\varphi\circ\Delta$. If $\varphi$ is a $K$-vector space isomorphism
we call $\varphi$ a $\Delta$-isomorphism.
\smallskip

A $\Delta$-module $M$ is called irreducible if $\{\v 0\}$ and $M$ are the
only $\Delta$-submodules.
\smallskip

The tensor product $M\otimes M'$ can be given a $\Delta$-module structure
via
$$\Delta(m\otimes m')=\Delta(m)\otimes\Delta(m'),$$
for all $m\in M,m'\in M'$.
\smallskip

The dual vector space $M^*$ can be given a $\Delta$-module structure via
$$\langle \Delta(m^*),m\rangle = \Delta(\langle m^*,\Delta^{-1}(m)\rangle).$$
Here $\langle m^*,m\rangle$ stands for the evaluation of $m^*$ in $m$.
\end{definition}

\begin{proposition}\label{explicitqdual}
Let $M$ be a finite dimensional $\Delta$-module over $K$. Let $m_1,\ldots,m_r$ be a
$K$-basis and let $A_{ij}\in K$ be such that
$$\Delta(m_i)=\sum_{j=1}^rA_{ij}m_j$$
for all $i$. Define the dual basis $m^*_i$ by $\langle m^*_i,m_j\rangle=\delta_{ij}$ for
$i,j=1,\ldots,r$, where $\delta_{ij}$ is the Kronecker delta. Let $(B_{ij})_{1\leq i,j\leq r}$
be the transposed inverse of $(A_{ij})_{1\leq i,j\leq r}$. Then,
\begin{enumerate}
\item $\Delta(m^*_i)=\sum_{j=1}^rB_{ij}m^*_j$ for $i=1,\ldots,r$,
\item $\omega=\sum_{i=1}^rm^*_i\otimes m_i\in M^*\otimes M$ does not depend
on the choice of basis $m_i$ and $\Delta(\omega)=\omega$.
\end{enumerate}
\end{proposition}

\begin{proof}
To prove the first assertion, let us define the matrix with entries $E_{ij}$ by
the relation $\Delta(m_i^*)=\sum_{i=1}^rE_{ij}m_j^*$. Combine this with
$m_k=\sum_{l=1}^rA_{kl}m_l$ to get
$$\delta_{ik}=\langle\Delta(m_i^*),\Delta(m_k)\rangle
=\sum_{j,l=1}^rE_{ij}A_{kl}\langle m_j^*,m_l\rangle$$
for any $i,k=1,\ldots,r$. Working out the right hand side yields
$\delta_{ik}=\sum_{j=1}^rE_{ij}A_{kj}$, which shows that $(E_{ij})_{1\leq i,j\leq r}$ is
the transposed inverse of $(A_{ij})_{1\leq i,j\leq r}$.

The second assertion then follows by straightforward computation.
\end{proof}

In the following proposition we consider elements $\Omega\in N\otimes M$
where $N,M$ are $\Delta$-modules of the same rank $r$. We say that $\Omega$ is
 non-degenerate if it cannot be written in the form $\sum_{i=1}^sn_i\otimes m_i$
with $m_i\in M,n_i\in N$ and $s<r$.

\begin{proposition}\label{proposition:deltaisomorphism}
Let $M,N$ be $\Delta$-modules of finite rank $r$. Let $m_1,\ldots,m_r$ be a
basis of $M$ and $m_1^*,\ldots,m_r^*$ the corresponding dual basis of $M^*$.

Then the
$\Delta$-homomorphisms $M^*\to N$ are in one-to-one correspondence with the tensors
$\Omega\in N\otimes M$ such that $\Delta(\Omega)=\Omega$.

The correspondence sends a $\Delta$-morphism
$\varphi:M^*\to N$ to $\Omega=\sum_{i=1}^r\varphi(m^*_i)\otimes m_i$.

Conversely, to a tensor $\Omega=\sum_{i=1}^r n_i\otimes m_i$ with $\Delta(\Omega)=\Omega$,
the $K$-linear map generated by $m^*_i\mapsto n_i$ is a $\Delta$-morphism from
$M^*$ to $N$.

Moreover, $\varphi$ is a $\Delta$-isomorphism if and only if $\Omega$ is non-degenerate.
\end{proposition}

\begin{proof}
Suppose we are given a $K$-linear map $\varphi:M^*\to N$.
It is determined by its values $n_i:=\varphi(m_i^*)$
for $i=1,\ldots,r$. We write the $\Delta$-homomorphism condition in two ways.

First of all, by definition one should have $\Delta(n_i)=\Delta(\varphi(m_i^*))=
\varphi(\Delta(m_i^*))$ for all $i$. Using Proposition \ref{explicitqdual} this is the
same as
$$\Delta(n_i)=\varphi\left(\sum_{j=1}^rB_{ij}m_j^*\right)=\sum_{j=1}^rB_{ij}n_j, \quad i=1,\dots,r.$$
Now define $\Omega=\sum_{i=1}^r n_i\otimes m_i$ and subsitute it into $\Delta(\Omega)=\Omega$.
We get
$$\Delta(\Omega)=\sum_{i=1}^r \Delta(n_i)\otimes \Delta(m_i)=\Omega.$$
Using $\Delta(m_i)=\sum_{j=1}^rA_{ij}m_j$ rewrite this as
$$\sum_{j=1}^r \left(\sum_{i=1}^r\Delta(n_j)A_{ij}\right)\otimes m_j=\Omega.$$
This is equivalent to $\sum_{i=1}^r\Delta(n_i)A_{ij}=n_j$ for $j=1,\ldots,r$.
Here we recognize our formulation of the $\Delta$-homomorphism condition.

Note that the tensor $\sum_{i=1}^rn_i\otimes m_i$ is non-degenerate if and
only if $n_1,\ldots,n_r$ are linearly independent. But this is equivalent to
$\varphi$ being an isomorphism.
\end{proof}

Let  ${\mathcal K}$ be a field extension of $K$
and suppose $\Delta$ extends to an isomorphism $\Delta:{\mathcal K}\to{\mathcal K}$.
Suppose also that the field of fixed elements under $\Delta$ is still $K_0$.
%(e.g $K=\bbbc(q,z)$ and ${\mathcal K}=\bbbc(q)((z))$ with $k=\bbbc(q)$).
Let $h_1,\ldots,h_r\in{\mathcal K}$. Define the Wronskian matrix by
$$W(h_1,\ldots,h_r):=\left(\begin{matrix}
h_1 & h_2 & \ldots & h_r\\
\Delta(h_1) & \Delta(h_2) & \ldots & \Delta(h_r)\\
\vdots & & & \vdots\\
\Delta^{r-1}(h_1) & \Delta^{r-1}(h_2) & \ldots & \Delta^{r-1}(h_r)\end{matrix}\right).$$

\begin{lemma}\label{lemma:wronskian}We have that $\det(W(h_1,\ldots,h_r))=0$ if and only if
$h_1,\ldots,h_r$ are linearly dependent over $K_0$.
\end{lemma}

\begin{proof}
When $h_1,\ldots,h_r$ are $K_0$-linear dependent the vanishing of the determinant  is
straightforward.

To prove the converse statement we proceed by induction on $r$. Suppose that
$\det(W(h_1,\ldots,h_r))=0$.
When $r=1$ we trivially get $h_1=0$.
Suppose $r>1$ and our statement holds for $h_1,\ldots,h_{r-1}$.

If $\rank_{\mathcal K}(W)<r-1$, then $W(h_1,\ldots,h_{r-1})$ has rank $<r-1$ and
the induction hypothesis implies that $h_1,\ldots,h_{r-1}$ are $K_0$-linear dependent.

Suppose $\rank_{\mathcal K}(W)=r-1$. Then there exist $\alpha_1,\ldots,\alpha_r\in{\mathcal K}$,
unique up to a common factor, such that $\sum_{i=1}^r\alpha_i\Delta^l(h_i)=0$ for
$l=0,1,\ldots,r-1$ (linear dependence of columns of $W$).

Let $s$ be the smallest index such that the $s$-th row is ${\mathcal K}$-linear dependent
of the previous rows. Then, by repeated applications of $\Delta$ we see that the
next rows are also dependent of the first $s-1$ rows. Since the rank of $W$ is $r-1$
we conclude that $s=r$ and there exist $\beta_0,\ldots,\beta_{r-2}\in{\mathcal K}$ such that
$\Delta^{r-1}(h_j)=\sum_{l=0}^{r-2}\beta_l\Delta^l(h_j)$ for $j=1,\ldots,r$. Apply $\Delta$ to obtain
$\Delta^r(h_j)=\sum_{l=0}^{r-2}\Delta(\beta_l)\Delta^{l+1}(h_j)$. It follows from this
that the above column relations also hold in case we take $l=r$.
Hence, after application of $\Delta^{-1}$,
$$\sum_{i=1}^r\Delta^{-1}(\alpha_i)\Delta^l(h_i)=0,$$
for $l=0,1,\ldots,r-1$. Since the coefficients of the column relations of $W$ are
unique up to a common factor, we find that there exists $\lambda\in{\mathcal K}$ such
that $\Delta^{-1}(\alpha_i)=\lambda\alpha_i$ for $i=1,2,\ldots,r$. Suppose that 
$\alpha_j\ne0$. Then we find that $\alpha_i/\alpha_j\in K_0$ for all $i$.
This gives us the desired $K_0$-linear relation
between the $h_i$'s.
\end{proof}

The next result identifies (up to $\Delta$-isomorphism) the dual difference operator associated with any fixed difference operator  $L\in K[\Delta,\Delta^{-1}]$ satisfying $A_0:=L(0)\neq0$.

\begin{theorem}\label{theorem:qdual}
Consider the $r$-th order difference operator
$$ L=A_r\Delta^r+A_{r-1}\Delta^{r-1}+\cdots+A_1\Delta+A_0,$$
with $A_i\in K$ for all $i$, and $A_r,A_0\neq0$. Then the dual of $K[\Delta,\Delta^{-1}]/(L)$ is $\Delta$-isomorphic
to $K[\Delta,\Delta^{-1}]/(L^*)$ where
$$L^*=\Delta^{r-1}(A_0)\Delta^r+\Delta^{r-2}(A_1)\Delta^{r-1}+
\cdots+A_{r-1}\Delta+\Delta^{-1}(A_r).$$
\end{theorem}

\begin{proof}Consider the element
\begin{multline*}
\Omega=-\Delta^{-1}(A_r)(1\otimes\Delta^{r-1})+A_{r-2}(\Delta\otimes\Delta^{r-2})\\+(1+\Delta)(A_{r-3}\Delta\otimes\Delta^{r-3})
+\dots+(1+\Delta+\cdots+\Delta^{r-2})(A_0\Delta\otimes1).
\end{multline*}
in $(K[\Delta,\Delta^{-1}]/(L^*))\otimes(K[\Delta,\Delta^{-1}]/(L))$.
Thanks to Proposition~\ref{proposition:deltaisomorphism},
it suffices to prove that $\Delta(\Omega)=\Omega$
and that $\Omega$ is non-degenerate. We compute
\begin{eqnarray*}
(1-\Delta)\Omega&=&-\Delta^{-1}(A_r)(1\otimes\Delta^{r-1})+A_r(\Delta\otimes\Delta^r)\\
&&+(1-\Delta)(A_{r-2}\Delta\otimes\Delta^{r-2})+(1-\Delta^2)(A_{r-3}\Delta\otimes\Delta^{r-3})\\
&&+\dots+(1-\Delta^{r-1})(A_0\Delta\otimes1)\\
&=&A_r\Delta\otimes\Delta^r+A_{r-2}\Delta\otimes
\Delta^{r-2}+\cdots+A_0\Delta\otimes1\\
&&-\Delta^{-1}(A_r)(1\otimes\Delta^{r-1})\\
&&-\Delta(A_{r-2})(\Delta^2\otimes\Delta^{r-1})-\cdots-(\Delta^{r-1}A_0)(\Delta^r\otimes\Delta^{r-1})\\
&=&\Delta\otimes L-L^*\otimes\Delta^{r-1}=0-0=0.
\end{eqnarray*}
Non-degeneracy of our $\Omega$ follows from the invertibility of the matrix $\Psi_q$ defined below, and is equivalent to $A_0,A_r\neq0$.
\end{proof}

\begin{remark}
Our choice of $L^*$ is by no means unique. As M.~van~der~Put pointed out to us,
$$\Delta\circ L^*\circ\Delta^{-1}=\Delta^r\circ A_0+\Delta^{r-1}\circ A_1+\cdots+\Delta\circ A_{r-1}+A_r$$
or rather,
$$\Delta^r\left(\Delta^{-r}\circ A_r+\Delta^{-(r-1)}\circ A_{r-1}+\cdots+\Delta^{-1}\circ A_1+A_0\right)$$
might be a more natural candidate. The subsequent arguments in this paper would also be simpler.
However, for some reason, this operator produces
a version of Theorem~\ref{qduality} with a matrix $(M_{kl}(q))_{0\leq k,l\leq r-1}$ having fewer zeros than the present matrix.
For this reason we have given preference to the more complicated $L^*$ defined above.
\end{remark}

We get the following immediate consequence.
\begin{corollary}\label{corollary:qomegafixed}
Let ${\mathcal K}$ be a $\Delta$-extension of $K$ and suppose that
$f,g\in {\mathcal K}$ satisfy the equations $L(f)=0$ and $L^*(g)=0$, where $L$ and $L^*$ are defined in Theorem~\ref{theorem:qdual}. Then
the element
\begin{eqnarray*}
\Omega(g,f)&:=&-\Delta^{-1}(A_r)g\Delta^{r-1}(f)\\
&&+A_{r-2}\Delta(g)\Delta^{r-2}(f)\\
&&+A_{r-3}\Delta(g)\Delta^{r-3}(f)+\Delta(A_{r-2})\Delta^2(g)\Delta^{r-2}(f)\\
&&\vdots\\
&&+A_0\Delta(g)f+\Delta(A_0)\Delta^2(g)\Delta(f)+\cdots+\Delta^{r-2}(A_0)\Delta^{r-1}(g)\Delta^{r-2}(f)
\end{eqnarray*}
belongs to $K_0$, the subfield of elements of ${\mathcal K}$ fixed under $\Delta$.
\end{corollary}

The element $\Omega(g,f)$ can be written in matrix form as follows:
$$(g,\Delta (g),\ldots,\Delta^{r-1}(g))
\left(\begin{matrix}0 & 0 & \cdots      & 0 &-\Delta^{-1}(A_r)\\
A_0 & A_1 & \cdots              &  A_{r-2} & 0\\
0 & \Delta (A_0) &   \cdots          & \Delta (A_{r-3}) & 0\\
\vdots &\vdots  & &\vdots &\vdots\\
0 & 0 &     \cdots                &  \Delta^{r-3}(A_1) & 0\\
0 & 0 & \cdots               & \Delta^{r-2}(A_0) & 0\end{matrix}\right)
\left(\begin{matrix}f\\ \Delta (f)\\ \vdots\\ \Delta^{r-1}(f)\end{matrix}\right).$$
Denote the middle $r\times r$-matrix, whose entries are in $K$, by $\Psi_q$.

Suppose we have a basis of solutions $f_1,\ldots,f_r$ of $L(f)=0$ and a basis
of solutions $g_1,\ldots,g_r$ of $L^*(g)=0$. Then, as $A_0,A_r\neq0$, the matrix $\Psi_q$ is invertible. Let us write $C_{ij}=\Omega(g_i,f_j)$
for all $i,j$ in $\{1,2,\ldots,r\}$ and denote the $r\times r$-matrix with these
entries by $C$. Letting $\bbbf:=W(f_1,\ldots,f_r)$ and $\bbbg:=W(g_1,\ldots,g_r)$, the previous matrix relation can be rewritten as
\begin{equation}\label{prematrixqduality}
\bbbg^t\Psi_q\bbbf=C.
\end{equation}
As a consequence of Lemma~\ref{lemma:wronskian} the matrices $\bbbf$ and $\bbbg$ are invertible and we get
\begin{equation}\label{matrixqduality}
\bbbf C^{-1}\bbbg^t=\Psi_q^{-1}.
\end{equation}
Notice that $\Psi_q^{-1}$ has zeros on and below the diagonal, except for the
element on place $r,1$. Moreover, by appropriately choosing the basis $g_1,\ldots,g_r$, we can assume that $C$ is the identity matrix, therefore yielding the following immediate consequence of~\eqref{matrixqduality}.
\begin{corollary}\label{qdiagonalidentity}
Let $f_1,\ldots,f_r$ be a basis of solutions of $L(f)=0$, where $L$ is defined in Theorem~\ref{theorem:qdual}. Then there
exists a solution basis $g_1,\ldots,g_r$ of $L^*(g)=0$ such that
$$\sum_{i=1}^r\Delta^k(g_i)\Delta^l(f_i)\in K$$
for all $k,l\ge0$. In particular we have $\sum_{i=1}^rg_i\Delta^l(f_i)=0$
for $l=0,1,\ldots,r-2$.
\end{corollary}

%%%%%%%%%%%%%%%%%%%%%%%%%%%%%%%%%%%%%%%%%%%%%%%%%
\section{Basic hypergeometric equation of order $r$}\label{qhypergeometric}
%%%%%%%%%%%%%%%%%%%%%%%%%%%%%%%%%%%%%%%%%%%%%%%%%

As an application of the previous section,  consider $\Delta=\sigma_q$ and the basic hypergeometric equation~\eqref{eq:qhyperg}, which we rewrite below:
$$(1-b_1\Delta/q)\cdots(1-b_{r}\Delta/q)f=z(1-a_1\Delta)\cdots(1-a_r\Delta)f.$$
By Theorem~\ref{theorem:qdual}, the dual equation reads
$$(\Delta-b_1/q)\cdots(\Delta-b_{r}/q)g=(\Delta-a_1)\cdots(\Delta-a_r)(z/q)g.$$
After rearranging factors we obtain
$$(1-q\Delta/b_1)\cdots(1-q\Delta/b_{r})g=\frac{a_1\cdots a_r}{b_1\cdots b_{r-1}}q^{r-2}z
(1-q\Delta/a_1)\cdots(1-q\Delta/a_r)g.$$
Note that this way of writing the dual general hypergeometric equation  was already  given in~\cite[Proposition~5]{Ro11}.
So the dual equation is again a hypergeometric equation with parameters
$a'_i=q/a_i,b'_i=q^2/b_i$ for $i=1,\ldots,r$ and $a_1\cdots a_rq^{r-2}z/(b_1\cdots b_{r-1})$
instead of $z$.

Suppose that none of the ratios $b_i/b_j$ with $i\ne j$ is an integer power of $q$.
Recall from the introduction that a basis of solutions of~\eqref{eq:qhyperg} reads
$$f_i(q;z)=z^{1-\beta_i}\,{}_r\phi_{r-1}\!\left[\begin{matrix}qa_1/b_i,\ldots,qa_r/b_i\\qb_1/b_i,\ldots,\vee,
\ldots,qb_r/b_i\end{matrix};q,z\right],\qquad 1\leq i\leq r.$$
Therefore a basis of solutions for the dual equation is given by
$$g_i(q;z)=z^{\beta_i-1}\,{}_r\phi_{r-1}\!\left[\begin{matrix}b_i/a_1,\dots,b_i/a_r\\qb_i/b_1,\dots,\lor,\dots,qb_i/b_r\end{matrix};q;\frac{a_1\dots a_rzq^{r-2}}{b_1\dots b_{r-1}}\right],$$
 as was defined in the introduction.

In a formal setup the ground field is now the field of rational functions $H_q(z)$ where $H_q$ is
the field $\bbbq$ extended with $q$ and the $a_i,b_j$. For the field ${\mathcal K}$ containing the solutions of the
difference equation we can take the field $H_q((z))$ of Laurent series with coefficients in  $H_q$ extended with the functions $z^{1-\beta_i}$,
which are defined as non-trivial solution of the difference equation $\Delta(f)=(q/b_i) f$. So loosely
speaking, $\beta_i=\log(b_i)/\log(q)$.

\begin{proposition}\label{dualpairingheine}
Let $\Psi_q$ and $C$ be the $r\times r$-matrices defined in Section~\ref{deltamodules},
but for the case of the q-hypergeometric equation of order $r$. For the solution bases of the equation
and its dual we take the bases just defined. Then $C$ is a diagonal matrix. The $i$-th diagonal
entry reads
$$C_{ii}={1\over qb_i^{r-2}}\prod_{j=1\atop j\neq i}^r(b_i-b_j),$$
where $b_r=q$.
\end{proposition}

\begin{proof} Let $H_q$ be the field generated over $\bbbq$ by the $a_i$'s, the $b_j$'s and $q$.
This is contained in the field $K_0$ of
constants. The elements of $C$ belong to $H_q$. In particular, $C_{ij}\in H_q$ for all $i\ne j$.
However, $C_{ij}$ is also of the form $z^{-\beta_i+\beta_j}$ times a power series in $z$.
Since $\beta_i\ne\beta_j$ if $i\ne j$ we can only conclude that $C_{ij}=0$ for all
distinct $i,j$. For the computation of $C_{ii}$ it suffices to look only at
constant terms after cancellation of the powers $z^{1-\beta_i}$ and $z^{\beta_i-1}$
in the evaluation of equation~\eqref{prematrixqduality}.
This means $C_{ii}$ equals
$$\left(1,b_i/q,(b_i/q)^2,\ldots,(b_i/q)^{r-1}\right)
\left(\begin{matrix}0 & 0 & \cdots     & 0 &-B_r\\
B_0 & B_1 &  \cdots              & B_{r-2} & 0\\
0 & B_0 &    \cdots          &B_{r-3} & 0\\
\vdots &\vdots & &\vdots  &\vdots\\
0 & 0 &     \cdots                 &  B_1 & 0\\
0 & 0 & \cdots              & B_0 & 0\end{matrix}\right)
\left(\begin{matrix}1\\ q/b_i\\ (q/b_i)^2\\ \vdots\\ (q/b_i)^{r-1}\end{matrix}\right),
$$
where the $B_j$'s are the constants in the coefficients of the q-hypergeometric
operator. That is,
$$B_0+B_1\Delta+\cdots+B_r\Delta^r=(1-b_i\Delta/q)\cdots(1-b_{r}\Delta/q).$$
Rewrite this as
$$B_0x^r+B_1x^{r-1}+\cdots+B_r=(x-b_1/q)(x-b_2/q)\cdots(x-b_r/q),$$
where $b_r=q$ and $B_0=1$. Note that this implies for any $i$ that
$$B_0(b_i/q)^r+B_1(b_i/q)^{r-1}+\cdots+B_{r-1}(b_i/q)+B_r=0,$$
and after differentiation of the polynomial,
$$rB_0(b_i/q)^r+(r-1)B_1(b_i/q)^{r-1}+\cdots+B_{r-1}=b_i(b_i-b_1)\cdots\vee\cdots(b_i-b_r)/q^r.$$
Subtraction of the first identity yields
\begin{multline*}
(r-1)B_0(b_i/q)^r+(r-2)B_1(b_i/q)^{r-1}\\
+\cdots+B_{r-2}(b_i/q)^2-B_r
=
b_i(b_i-b_1)\cdots\vee\cdots(b_i-b_r)/q^r.
\end{multline*}
But evaluation of $C_{ii}$ gives
$$C_{ii}=(q/b_i)^{r-1}(-B_r+(r-1)B_0(b_i/q)^r+\dots+B_{r-2}(b_i/q)^2).$$
In view of the previous calculation this implies our assertion.
\end{proof}

As an immediate consequence of Corollary~\ref{qdiagonalidentity} and Proposition~\ref{dualpairingheine}, we get the following result.
\begin{corollary}\label{finalheine}
Let $(f_i)_{1\leq i\leq r}$ and $(g_j)_{1\leq j\leq r}$ be the basis of solutions of the q-hypergeometric equation and the
dual equation, as given above. Let $H_q$ be the field generated over $\bbbq$ by the $a_i,b_j$ and $q$. Then, with
$C_{ii}$ as in the previous theorem we get
$$\sum_{i=1}^r{1\over C_{ii}}\Delta^k(f_i)\Delta^l(g_i)\in H_q(z),$$
for all integers $k,l\ge0$. In particular we have $\sum_{i=1}^r{1\over C_{ii}}\Delta^k(f_i)g_i=0$
for $k=0,1,\ldots,r-2$.
\end{corollary}
Finally, to prove the cases $k=r-1,l=0$ and $l=k+1$ of Theorem~\ref{qduality}, it remains to explicitly invert the matrix $\Psi_q$, and to use~\eqref{matrixqduality} and Proposition~\ref{dualpairingheine}.

\ifx
Directly from the definition of $\Psi$, we have
\begin{equation*}\label{inverse}
\Psi^{-1}=\left(
\begin{matrix}
0&x_0&x_1&\dots&x_{r-2}\\
&&&&\\
0&0&\Delta x_0&\dots&\Delta x_{r-3}\\
\vdots&\vdots&\vdots&&\vdots\\
0&0&0&\dots&\Delta^{r-2} x_{0}\\
&&&&\\
-\Delta^{-1}(1/A_r)&0&0&\dots&0
\end{matrix}\right),\end{equation*}
where $x_0=1/A_0=1/(1-z)$ and for all $r\geq3$, $\sum_{j=0}^{r-2}x_j\Delta^j(A_{r-2-j})=0$. As $\Delta^j(A_{r-2-j})\in h(z)$, this gives the desired recursive definition of the $x_j$'s.\\

Thanks to the previous expression of $\Psi^{-1}$, note that there are only $2r-1$ different duality relations encoded in~\eqref{matrixduality}, all other $r^2-2r+1$ identities being consequences of these ones.
\fi
%%%%%%%%%%%%%%%%%%%%%%%%%%%%%%%%%%%%%%%%%%%%%%%%%%%%%%%%%%%%%%%%

\end{document}